\documentclass[11pt,reqno,a4paper]{amsart}
\usepackage{tikz}
\usetikzlibrary{matrix}
\usepackage{amsfonts}
\usepackage{epsfig}
\usepackage{graphicx}
\usepackage{amsmath}
\usepackage{amssymb}
\usepackage{color}
\usepackage{mathrsfs}
\usepackage{txfonts}
\usepackage{mathtools} 
\usepackage[normalem]{ulem}

\usepackage{epsfig}
\usepackage{color}
\usepackage{amssymb,amsmath,amsthm,amstext,amsfonts}
\usepackage{psfrag}
\usepackage{url}
\usepackage{epstopdf}
\usepackage{epic,eepic}
\usepackage{upgreek}
\usepackage{amssymb}
\usepackage{mathrsfs}

\makeatletter \@addtoreset{equation}{section} \makeatother

\renewcommand\thetable{\thesection.\@arabic\c@table}

\theoremstyle{plain}
\newtheorem{maintheorem}{Theorem}
\newtheorem{maincorollary}{Corollary}

\newtheorem{theorem}{Theorem}[section]

\newtheorem{lemma}{Lemma}[section]

\newtheorem{remark}{Remark}[section]
\newtheorem{example}{Example}[section]

\newcommand{\vep}{\varepsilon}

\newcommand{\diam}{\operatorname{diam}}
\newcommand{\dist}{\operatorname{dist}}

\newcommand{\cC}{\mathcal{C}}

\newcommand{\cH}{\mathcal{H}}

\newcommand{\cL}{\mathcal{L}}

\newcommand{\cF}{\mathcal{F}}

\newcommand{\cU}{\mathcal{U}}

\newcommand{\cA}{\mathcal{A}}

\newcommand{\Homeo}{\text{Homeo}(M)}

\newcounter{main}

\usepackage[colorlinks=true, linkcolor=blue, urlcolor=red, citecolor=blue, hyperindex]{hyperref}

\title{On sensitivity to initial conditions and uniqueness  of \\ conjugacies for structurally stable diffeomorphisms}

\author[Jorge Rocha]{Jorge Rocha}
\address{Jorge Rocha, Departamento de Matem\'atica, Universidade do Porto, 
Rua do Campo Alegre, 687, 
4169-007 Porto, Portugal}
\email{jrocha@fc.up.pt}

\author[Paulo Varandas]{Paulo Varandas}
\address{Paulo Varandas, Departamento de Matem\'atica, Universidade Federal da Bahia\\
Av. Ademar de Barros s/n, 40170-110 Salvador, Brazil}
\email{paulo.varandas@ufba.br}

\begin{document}

\begin{abstract}
In this paper we study $C^1$-structurally stable diffeomorphisms, that is, $C^1$ Axiom A diffeomorphisms with the 
strong transversality condition.  In contrast to the case of dynamics restricted to a hyperbolic basic piece, structurally
stable diffeomorphisms are in general not expansive and the conjugacies between $C^1$-close structurally 
stable diffeomorphisms may be non-unique, even if there are assumed $C^0$-close to the identity. 
Here we give a necessary and sufficient condition for a structurally stable diffeomorphism to admit a 
dense subset of points with expansiveness and sensitivity to initial conditions.
Morever, we prove that the set of conjugacies between elements in the same conjugacy class is
homeomorphic to the $C^0$-centralizer of the dynamics. Finally, we use this fact to deduce that 
any two $C^1$-close structurally stable diffeomorphisms
are conjugated by a unique conjugacy $C^0$-close to the identity if and only if these are Anosov. 
\end{abstract}

\date{\today}
\maketitle
 
\hfill{Dedicated to Welington de Melo }

\section{Introduction}

One of the leading problems considered by the dynamical systems community
has been to provide a global view of the space of dynamical systems. In fact, based on the pioneering works of 
Andronov, Pontryagin, Peixoto and Smale, in the nineties Palis proposed a conjecture that roughly describes the complement of uniform hyperbolicity 
as the space of diffeomorphisms that are approximated by those exhibiting either homoclinic tangencies or heteroclinic cycles. This program has been carried out with success in the $C^1$-topology, where perturbation tools as the Pugh closing lemma, the Franks' lemma, the Hayashi's connecting lemma or the Ma\~n\'e's ergodic closing lemma are available (see e.g.~\cite{Pugh, Ma1, Hay} and references therein). 
Uniform hyperbolicity helped to coin the idea of stability that is,  to characterize the dynamics that persist and behave similarly under small perturbations.  The stability theorem for isolated hyperbolic sets asserts that  
any isolated hyperbolic basic set $\Lambda$  for a $C^1$-diffeomorphism $f$ admits a continuation 
$\Lambda_g$ for any $C^1$-small perturbation $g$ 
of the original dynamics, whose dynamics is topologically conjugate to $f\mid_{\Lambda_f}$: there exists a unique
homeomorphism $h_g : \Lambda_f \to \Lambda_g$ that is $C^0$-close to the identity and so that $h_g\circ f
= g \circ h_g$. The uniqueness of the conjugacy $h_g$ at finite distance to the identify reflects the rigidity of hyperbolicity
and can be obtained, via Banach fixed point theorem, as a fixed point for the operator 
$L(h)=  g^{-1}\circ h\circ f$ acting on the space $\Homeo$ of homeomorphisms in a compact manifold $M$.

A question that arises naturally in the vein of the stability for uniformly hyperbolic dynamics is to understand if 
hyperbolicity is a necessary and sufficient condition to characterize structurally stable diffeomorphisms. 
Recall that a $C^1$-diffeomorphism $f$ is called \emph{structurally stable} if there exists an open neighborhood $\cU$ of $f$ in the $C^1$-topology such that for any $g\in \cU$ there exists a homeomorphism $h_g: M \to M$ such that $h_g$ conjugates the dynamics, that is, $h_g \circ f = g \circ h_g$. After the first examples of $\Omega$-explositions there was a clear idea 
that heteroclinic cycles and tangencies constitute obstructions to structural stability and that 
uniform hyperbolicity should play a key role (see e.g.~\cite{Sm70}).   
Robbin, Robinson and Ma\~n\'e \cite{Robbin72, Robinson,Ma1} completed the proof that $C^1$-structural stability is equivalent to the Axiom A and the strong transversality conditions (in dimension two this was obtained by de Melo~\cite{deMelo}). In particular, if $f$ is structurally stable, then there exists a 
$Df$-invariant decomposition $T_{\Omega(f)}M= E^s\oplus E^u$ and constants $C>0$ and $\lambda\in (0,1)$
such that 
$$
\|Df^n(x) \mid_{E^s_x}\| \le C \lambda^n
	\quad\text{and}\quad
	\|Df^{-n}(x) \mid_{E^u_x}\| \le C \lambda^n
$$
for every $n\ge 1$ and every $x$ in the non-wandering set $\Omega(f)$, and 
$\Omega(f)$ coincides with the closure of the periodic points. 
Moreover, the non-wandering set of a structurally stable diffeomorphism can be decomposed as a finite set of 
hyperbolic sets whose basins of attraction cover the entire manifold. A priori, such decomposition could suggest 
that the rigidity of the conjugacies on each hyperbolic basic set would `spread' to the manifold or, in other words, 
that the global conjugacy could be completely determined by the conjugacies of each of the hyperbolic basic pieces.

Our purpose here is to revisit structurally stable diffeomorphisms, discussing both expansiveness properties
as their space of conjugacies given, as starting point, the (local) $\Omega$-stablity of hyperbolic basic pieces. 
In the context of structural stability the notion of expansiveness is associated to a very 
rigid phenomenon, because a $C^1$-structurally stable diffeomorphism is expansive if and only if it is 
Anosov~\cite{Ma75}. 
In the context of  compact surfaces, all expansive homeomorphisms on surfaces are conjugate to Anosov 
diffeomorphisms  if $M=\mathbb T^2$ and are pseudo-Anosov if the genus is larger or equal to two \cite{Lewowicz,Hiraide}. 
Thus, $\mathbb S^2$ admits structurally stable diffeomorphisms but admits no expansive homeomorphisms. 
In view of the later result it is natural to ask wether subsets of points with some expansiveness can be topologically large. 
Although there are structurally stable diffeomorphisms that admit no dense subset of points with expansiveness (see e.g.
Example~\ref{ex:N-S}) 
we prove that all structurally stable diffeomorphisms whose  topological basins of trivial attractors and repellers in $\Omega(f)$ do not intersect admit a dense subset of  points with expansiveness (cf. Theorem~\ref{prop:expansive.points}).   In rough terms, dense expansiveness holds if and only if some north-pole south-pole kind of dynamics cannot be embedded in the original dynamical system, and this condition implies
the sensitivity to initial conditions. 

The second part of our work concerns the study of the set of conjugacies appearing naturally in the context of structural stability.
 If $f$ is a structurally stable diffeomorphism there exists an open neighborhood $\cU$ of $f$ in the $C^1$-topology such that for any  $g\in \cU$ there exists a homeomorphism 
$h_g: M \to M$ such that $h_g$ conjugates both dynamics, that is, $h_g \circ f = g \circ h_g$. 
It is natural to ask under which conditions this conjugacy can be taken unique $C^0$-close to the identity.
Our approach to this problem is to relate the set of all such conjugacies with the $C^0$-centralizer of the dynamics 
by proving that for any diffeomorphism $g$ in the conjugacy class of a given diffeomorphism $f$ the space $\cH_{f,g}$ 
of conjugacies between $f$ and $g$ is homeomorphic to the $C^0$-centralizer 
$
Z^0(f)=\{h\in \Homeo : h\circ f = f\circ h\}
$ 
of $f$ (cf. Theorem~\ref{prop:bijection}). 
This leads to the comprehension of the $C^0$-centralizer of $C^1$-structurally stable diffeomorphisms.
In his seminal paper \cite{Sm90}, Smale asked whether `most' dynamical systems 
would have trivial $C^1$-centralizer. Although this is not yet completely understood, there are evidences that
Smale's question has an affirmative answer and some important contributions in that direction include e.g.~\cite{Ko70,PY89, 
Bu04,BCW, Fi09, BF}. 
From the purely topological viewpoint, one cannot expect an affirmative answer to Smale's question to hold.
Indeed, $C^0$-centralizers are larger and there are open sets of surface $C^1$-Anosov diffeomorphisms whose 
$C^0$-centralizer is discrete but non-trivial~\cite{Ro05}.
We refer the reader to Section~\ref{sec:prelim} for precise definitions and a more detailed discussion.
We include some examples (see Examples~\ref{ex:N-S} and ~\ref{ex:2}) to illustrate that even in the absence of
trivial basic pieces of the non-wandering set the conjugacies may contain a continuum of homeomorphisms
$C^0$-close to the identity. 
We prove that the $C^0$-centralizer of a structurally stable diffeomorphism is discrete if and only if it is an Anosov
diffeomorphism (cf. Theorem~\ref{abstractthm}). As a consequence, we deduce that the Ma\~n\'e characterization 
is optimal: there exist structurally stable diffeomorphisms with an invariant and dense set of points with expansiveness which are not Anosov (cf. Example~\ref{ex:2}).  

The paper is organized as follows. 
In Section~\ref{sec:prelim} we recall the notions and some results about uniformly hyperbolic sets, 
homoclinic classes and 
structural stability.  In Section~\ref{sec:statements} we state our main results. The proofs of the two main results will appear from Sections ~\ref{sec:dense} to ~\ref{sec:uniqueness}. Section~\ref{sec:examples} is devoted to examples.

\section{Preliminaries}\label{sec:prelim}

In this section we recall some results
about uniformly hyperbolic sets, homoclinic classes and structurally stable diffeomorphisms.
Throughout, let $M$ be a compact Riemannian manifold, let $d$ denote the Riemannian distance on $M$,
and let $\text{Diff}^{\, r}(M)$ ($r\ge 1$) denote the space of $C^r$-diffeomorphisms on $M$.

\subsubsection*{Hyperbolic sets, expansiveness and sensitivity to initial conditions}

Given $f \in \text{Diff}^{\, r}(M)$ ($r\ge 1$), let ${\text{Per}(f)}$ denote the set of periodic points for $f$ and 
let $\Omega(f) \subset M$ denote the non-wandering set of $f$. 
We say that a compact $f$-invariant set $\Lambda \subset M$ is a \emph{uniformly hyperbolic} for $f$ if
there exists a $Df$-invariant splitting $T_{\Lambda} M = E^s \oplus E^u$ and constants $C>0$ and $\lambda\in (0,1)$ 
so that 
$
\| Df^n(x)\mid_{E^s_x} \| \le C \lambda^n 
$
and
$
	\| (Df^n(x)\mid_{E^u_x})^{-1} \| \le C \lambda^n 
$
for every $x\in \Lambda$ and $n\ge 1$. 
A periodic point $p\in \text{Per}(f)$ is hyperbolic if $\mathcal O(p)=\{p, f(p), \dots, f^{n-1}(p)\}$ is a 
hyperbolic set for $f$, where $n$ stands for the period of $p$. Given
an $f$-invariant set $\Lambda$, we say that $f\mid_\Lambda$ is \emph{transitive} if it has a dense orbit, that is, there exists
$x\in \Lambda$ so that $\mathcal O (x) :=  \{  f^n(x) : n\in \mathbb Z \}$ is dense in $\Lambda$.
Given a point $x\in M$ and $\vep>0$ the \emph{$\vep$-stable set} of $x$ is defined by
$W_\vep^s(x) =\left\{y\in M : d(f^{n}(y),f^{n}(x))\le \vep \text{ for all } n \ge 0\right\}.$
Similarly,  $W_\vep^u(x) =\left\{y\in M : d(f^{-n}(y),f^{-n}(x))\le \vep \text{ for all } n \ge 0\right\}$
is the $\vep$-\emph{unstable set} of $x$.
Given a hyperbolic set $\Lambda$ for $f$ there exists a uniform $\vep>0$ so that the stable and unstable sets 
$W_\vep^s(x)$ and $W_\vep^u(x)$ are $C^r$ submanifolds tangent to $E^s_x$ and $E^u_x$, respectively, 
for every $x\in \Lambda$. These are called, respectively, the \emph{local stable} and
\emph{local unstable manifolds} at $x$ of size $\vep$. 
A subset $\Delta \subset W^s(x)$ is called a \emph{fundamental domain} 
if for any $z\in W^s(x) \setminus \{x\}$ there exists a unique $n=n(z) \in \mathbb Z$ so that $f^n(z) \in \Delta$.
We refer the reader to \cite{Shub} for more details.

Given $\vep>0$ we say that $\Lambda\subset M$ is a set of \emph{$\vep$-expansiveness} for $f$ if 
for any $x \neq y\in \Lambda $ there exists $n\in \mathbb Z$ so that $d(f^n(x), f^n(y))>\vep$.
Observe that no $f$-invariance condition is required.
When no confusion is possible 
we will refer to these simply as sets of expansiveness. 
We say that $f$ is \emph{densely expansive} if there exists an $\vep$-expansive 
dense set  in $M$, for some $\vep>0$. 
An homeomorphism $f$ is \emph{expansive} if the manifold $M$ is a set of expansiveness for $f$.
Expansive homeomorphisms in compact manifolds of dimension smaller or equal to $2$ are either 
Anosov or pseudo-Anosov (see e.g. ~\cite{Lewowicz, Hiraide} and references therein). 
Thus, there are geometrical and topological obstructions for a manifold to admit expansive dynamics.
Although there exists no complete classification of expansive homeomorphisms in dimension 
larger than $2$, very important contributions in this direction have been obtained by Vieitez (see~\cite{Vieitez93, Vieitez02} and references therein).

Finally we recall the notion of sensitivity to initial conditions. We say that $f$ has \emph{sensitivity to initial conditions}
if there exists $\vep>0$ so that every $x\in M$ is accumulated by a sequence of points $(x_n)_n$ in $M$ such that 
$\sup_{k\in \mathbb Z} d(f^k(x), f^k(x_n))>\vep$. This condition, weaker than expansiveness does not hold for all
structurally stable diffeomorphisms (e.g. north-pole south-pole dynamics).

\subsubsection*{$C^r$-centralizers}

Given $f \in \text{Diff}^{\, r}(M)$, $r\in \mathbb N \cup\{\infty\}$ and $0\le k \le r$, the \emph{$C^k$-centralizer} for $f$ is the set
$
Z^k(f) = \{ g\in \text{Diff}^k(M) \colon g\circ f = f\circ g \},
$
where $\text{Diff}^0(M)$ stands for the space $\Homeo$. 
For every $0\le k \le n$, it is clear that $Z^k(f)$ is a subgroup of $(\text{Diff}^k(M), \circ)$
and it contains the subgroup $\{ f^n \colon n\in\mathbb Z \}$.
Clearly, the following inclusion holds
$Z^0(f) \supset Z^1(f) \supset Z^2(f) \supset \dots \supset  Z^r(f) \supset \{ f^n \colon n\in\mathbb Z \}.$
We say that $f \in \text{Diff}^{\, r}(M)$ has \emph{trivial $C^k$-centralizer} if the centralizer is the smallest possible, that is,
$
Z^k(f)=\{ f^n \colon n\in\mathbb Z \}.
$
We say that the diffeomorphism $f \in \text{Diff}^{\, r}(M)$ has \emph{discrete $C^k$-centralizer} if there is $\vep>0$ so that $d_0(h_1,h_2)>\vep$ for every distinct $h_1, h_2 \in Z^k(f)$.
Walters~\cite{Wa70} proved that expansive homeomorphisms have discrete $C^0$-centralizers.
The centralizer of Komuro expansive flows (which include the Lorenz attractor) and $\mathbb R^d$ actions
on compact manifolds was described in \cite{BoRoVa}.

\subsubsection*{Axiom A diffeomorphisms and homoclinic classes}

In what follows we collect some results on Axiom A diffeomorphisms (proofs and more details can be found 
in ~\cite{Shub}). We say that a diffeomorphism $f \in \text{Diff}^{\, r}(M)$, $r\ge 1$, is \emph{Axiom A}  if (i) $\overline{\text{Per}(f)}=\Omega(f)$
and (ii) $\Omega(f)$ is a uniformly hyperbolic set.
Clearly all periodic points of Axiom A diffeomorphisms are hyperbolic. We say that $f$ is an \emph{Anosov diffeomorphism} 
if the whole manifold $M$ is a hyperbolic set for $f$.
If $p \in {\text{Per}(f)}$ is a hyperbolic periodic point for $f$ of period $k\ge 1$ 
then there exists a $C^1$-neighborhood $\cU$ of $f$ and a neighborhood $V\subset M$ of $p$ so that any 
$g\in \cU$ admits a unique hyperbolic periodic point $p(g)\in V \cap \text{Per}(g)$ of period $k$, referred as the \emph{continuation} of $p$. 
Given a hyperbolic periodic point $p$, 
the \emph{homoclinic class $H(p,f)$} for $p$ with respect to $f$ is defined by
$
H(p,f)= \overline{ W^u( \mathcal O(p) ) \pitchfork W^s(\mathcal O(p) ) }.
$
By the spectral decomposition theorem, 
for any Axiom A diffeomorphism $f$ there are finitely many (hyperbolic) periodic points $p_\ell$ ($1\le \ell \le k$)
and hyperbolic homoclinic classes $\Lambda_\ell=H(p_\ell,f)$ so that
$
\Omega(f) = \Lambda_1 \cup \Lambda_2 \cup \dots \cup \Lambda_k
$ 
(\emph{spectral decomposition}).
For every  $C^1$-small perturbation $g$ of $f$ the hyperbolic set $H(p_\ell,f)$ admits a 
continuation $H(p_\ell(g),g)$.  
Consider the partial order on the space $\{\Lambda_i : 1\le i \le k\}$ given by 
$\Lambda_i \succ \Lambda_j$ if and only if $W^u(\Lambda_i) \cap W^s(\Lambda_j) \neq \emptyset$.
We say that $f$ has \emph{no cycles} if whenever $\Lambda_{i_s} \succ \dots \succ \Lambda_{i_1}$ and 
$\Lambda_{i_1} = \Lambda_{i_s}$ then $s=1$.
Moreover, if $f$ is Axiom A with no cycles then there exists a \emph{filtration} adapted to $\Omega(f)$: there exists a nested sequence 
$\emptyset = M_0 \subset M_1 \subset \dots \subset M_k=M$ of smooth codimension 0 submanifolds with boundary
so that $f(M_i) \subset \text{int}(M_i)$ and $\bigcap_{n\in \mathbb Z} f^n(M_{i} \setminus M_{i-1})$
consists of a finite union of basic pieces of $\Omega(f)$ for every $1\le i \le k$.
We say that an $f$-invariant, compact and transitive set $\Lambda$ is an \emph{attractor} if there exists an open neighborhood 
$U$ of $\Lambda$ so that $\Lambda= \bigcap_{n\ge 0} f^n(\overline U)$. An $f$-invariant set is a \emph{repeller} 
if it is an attractor for $f^{-1}$. Finally we say that a basic piece of the non-wandering set is trivial if it consists of 
a periodic orbit. Given an attractor $\Lambda$ for $f$, the topological basin of attraction is the set $B(\Lambda)=\{ x\in M : \dist(f^{n} x, \Lambda) \to 0 \; \text{as} \; n\to\infty\}$. The topological basin of repulsion for some repeller $\Lambda$
is defined as the topological basin of attraction for $\Lambda$ with respect to $f^{-1}$.

\subsubsection*{Conjugacy classes and structural stability}

Given a $C^r$-diffeomorphism $f$, $r\ge 1$, we define the \emph{conjugacy class} of the diffeomorphism $f$ as the set 
$\mathcal{C}_f =\{ g\in \text{Diff}^{\,r}(M) \colon \text{ there}$ $\text{exists } h\in \Homeo \text{ s.t. } h\circ g = f\circ h \}.$ 
Consider  $\cH_f = \bigcup_{g\in \cC_f} \cH_{f,g}$ where 
$\cH_{f,g}=\{ h\in \Homeo : h\circ f = g\circ h\}$. 
In other words, $\cH_{f,g}$ 
denotes the space of all homeomorphisms that conjugate $f$ and $g$. 
Endow the space of homeomorphisms $\Homeo$ on $M$ with the distance $d_0$ defined by
$$
d_0( h_1,h_2) = \sup_{x\in M} d( h_1(x), h_2(x) ) + \sup_{x\in M} d( h_1^{-1}(x), h_2^{-1}(x) )
$$
for all $h_1,h_2 \in \Homeo$. Given $f \in \text{Diff}^{\, 1}(M)$, we say that $f$ is \emph{structurally stable} if
there exists an open neighborhood $\cU$ of $f$ in the $C^1$-topology such that for any $g\in \cU$ 
there exists a homeomorphism $h_g: M \to M$ such that $h_g$ conjugates the dynamics, 
that is, $h_g \circ f = g \circ h_g$.  
Robbin, Robinson and Ma\~n\'e~\cite{Robbin71,Robinson, Ma1} proved that a $C^1$-diffeomorphism is structurally stable if and only if it is Axiom A and satisfies the \emph{strong transversality condition}: if 
$
E^\pm_x = \{ v\in T_x M \colon \lim_{n\to\infty} \| Df^{\mp n}(x) v\|= 0  \}
$
then $T_xM = E^-_x + E^+_x$ for every $x\in M$.
Moreover, if $f$ is an Axiom A diffeomorphism then the strong transversality
condition is equivalent to the transversal intersection of every stable and unstable manifolds~\cite[Proposition~7.5]{Robbin71}.  In the proof of the first part of the stability conjecture, Robbin~\cite{Robbin71}, Robinson~\cite{Robinson} used the strategy developed by Moser~\cite{Moser}
(in the proof of the stability of Anosov diffeomorphisms) to construct conjugacies that vary continuously with the dynamical system:

\begin{theorem}\label{thm:RR} (Robbin, Robinson)
Let $f$ be a $C^1$ Axiom A diffeomorphism with the strong transversality condition. Then there exists an open neighborhood
$\mathcal U$ of $f$ and $K>0$ such that, for any $g\in \mathcal U$: 
\begin{enumerate}
\item[(i)] \emph{there exists} $h=h_g \in \Homeo$ so that $h_g \circ f = g\circ h_g$; and
\item[(ii)] $d_{C^0}(h_g, id) \le K d_{C^0}(f,g)$.
\end{enumerate}
\end{theorem}

The proof of the previous theorem relies on Banach's fixed point theorem 
for a family of operators $\cL_g$ 
that vary continuous with the diffeomorphism $g$ in 
a $C^1$-neighborhood $\cU$
of $f$, but whose construction depends on some fixed neighborhoods of the basic sets and a partition of unity for 
$M$ (cf. \cite{Robbin71}).
For that reason the selected conjugacies can be chosen to depend continuously with respect to the dynamics 
but uniqueness is not guaranteed. In fact, uniqueness of conjugacies $C^0$-close to the identity may fail 
(cf. Example~\ref{ex:N-S}). This is in a strong contrast
with the fact that the conjugacy restricted to each basic piece of the non-wandering set is unique and H\"older continuous.
Finally, for completeness we observe that if $f\in\text{Diff}^{\,2}(M)$ is $C^1$-structurally stable then there exists a $C^1$-open neighborhood $\cU$ of $f$ so that the conjugacy map $\mathcal U \ni g \mapsto h_g \in \Homeo$ can be chosen
to be $C^1$-differentiable \cite[Theorem~2]{Franks}.

\section{Statement of the main results}\label{sec:statements}

This section is devoted to the statement of our main results.

\subsection*{Expansiveness}

First we provide a necessary and sufficient condition for a structurally stable diffeomorphism 
to be densely expansive.

\begin{maintheorem}\label{prop:expansive.points}
Let $f\in \text{Diff}^{\,1}(M)$ be a structurally stable diffeomorphism. The following are equivalent: 
\begin{enumerate}
\item the topological basins of trivial attractors and repellers in $\Omega(f)$ do not intersect;
\item $f$ is densely expansive (i.e. there exists $\vep>0$ and a dense subset $D\subset M$ such that for 
any $x \neq y \in D$  there exists $n\in \mathbb Z$ satisfying $d(f^n(x), f^n(y))>\vep$);
\item $f$ has sensitivity to initial conditions.
\end{enumerate}
\end{maintheorem}

One should mention that whenever the topological basins of trivial attractors and repellers in $\Omega(f)$ do not intersect,
in the proof of Theorem~\ref{prop:expansive.points}, we construct a dense set of expansiveness $D \subset M$ 
which is not $f$-invariant and its saturated set $\bigcup_{n\in\mathbb Z} f^n(D)$ is not necessarily expansive. 

\subsection*{Conjugacy classes and $C^0$-centralizers}\label{sec:conju}

The following simple result builds a bridge 
between the $C^0$-centralizer of a dynamical system $f$ and set of conjugacies between diffeomorphisms in the same conjugacy class.

\begin{maintheorem}\label{prop:bijection}
Let $f \in \text{Diff}^{\,r}(M)$, $r\ge 0$. For every $g\in \cC_f$ and $h\in \mathcal H_{f,g}$ the map 
$$
\begin{array}{rccc}
F_h \colon &  Z^0(f) & \to & \mathcal H_{f,g} \\\medskip 
	& \tilde f & \mapsto & h \circ \tilde f
\end{array}
$$
is a homeomorphism and satisfies $F_h(\tilde f \circ \hat f ) = F_h(\tilde f) \circ h^{-1} \circ F_h(\hat f)$ for every
$\tilde f, \hat f \in Z^0(f)$. In particular,
\begin{itemize}
\item[(i)] $\mathcal H_{f,g}$ is homeomorphic to $Z^0(f)$ for every $g \in \cC_f$;
\item[(ii)] $\mathcal H_{f}$ is homeomorphic to $\cC_f \times Z^0(f)$;
\item[(iii)] $\mathcal H_{f,g}$ is a discrete subset of $\Homeo$ if and only if $Z^0(f)$ is a discrete 
subgroup of $\Homeo$. 
\end{itemize}
\end{maintheorem}

Since the map $F_h$ is a homeomorphism then the cardinality and topological properties of
all sets $\cH_{f,g}$ coincide for every $g\in \mathcal C_f$. This motivates the following
definition. Given $r\ge 0$, let $N^0(f) \in \mathbb N \cup\{\infty\}$ denote the minimum number of generators for 
the subgroup $Z^0(f)$. We are interested in studying the regularity of the following function:
$$
\begin{array}{rccc}
N^0  & \colon  \text{Diff}^{\,r}(M) & \to & \mathbb N\cup\{\infty\} \\
	& f & \mapsto & N^0(f)
\end{array}
$$
The following is an immediate consequence of Theorem~\ref{prop:bijection}:

\begin{maincorollary}
$N^0(\cdot)$ is a topological invariant: if $f$ and $g$ are topologically conjugate then $N^0(f)=N^0(g)$.
In particular, $N^0(\cdot)$ is a locally constant function restricted to the space of structurally stable diffeomorphisms.
\end{maincorollary}

As one could expect, the minimal cardinality $N^1(\cdot)$ of generators for the $C^1$-centralizer is not a topological
invariant even among Anosov diffeomorphisms. Indeed, since $C^1$-generic Anosov diffeomorphisms have trivial
centralizer then $N^1(f)=1$ for a $C^1$-generic set in $\text{Diff}^1(M)$. On the other hand, there are $C^\infty$-Anosov
diffeomorphisms with a discrete but non-trivial $C^1$-centralizer~\cite{Plykin}. 
If $f$ is structurally stable then Theorem~\ref{prop:bijection} implies that the uniqueness of 
conjugacies $C^0$-close to the identity is equivalent to the $C^0$-centralizer to be discrete. 
This, together with Theorem~\ref{thm:RR}, yields the following consequence.

\begin{maincorollary}\label{cor:structuralstab}
Let $f\in \text{Diff}^{\,1}(M)$ be a structurally stable diffeomorphism. If $Z^0(f)$ is discrete then
there exists $\vep>0$ and an open neighborhood $\mathcal U \subset \text{Diff}^{\,1}(M)$ of $f$ so that:
\begin{itemize}
\item[(a)] for any 
	$g\in \cU $ there exists a \emph{unique} conjugacy $h_g \in \mathcal H_{f,g}$ (that is, $h_g \circ f = g\circ h_g$) 
	which is $\vep$-$C^0$-close to the identity;
\item[(b)] the function $\mathcal U \ni g \mapsto h_g \in \Homeo$ is continuous;
\item[(c)] there exists $K>0$ so that $d_{C^0}(h_g,id) \le K d_{C^0}(g,f)$; and
\item[(d)] for any $g\in \cU$  the map $\Psi_{f,g} \colon   Z^0(f)  \to   Z^0(g)$ given by $\Psi_{f,g}(\tilde f)=h_g^{-1} \circ \tilde f \circ h_g$
is a homeomorphism.
\end{itemize}
\end{maincorollary}

The continuous dependence of the conjugacy map $h_g$ on the diffeomorphism $g$ follows from the work of Robbin 
and Robinson~\cite{Robbin72, Robinson}. 
In what follows we relate expansiveness to $C^0$-centralizers for 
structurally stable diffeomorphisms. Our starting point is the following result due to Walters \cite{Wa70}:

\begin{lemma}
\label{thm:C0-homoclinic}
Take  $r\ge 0$. Assume $f \in \text{Diff}^{\,r}(M)$ has an $f$-invariant subset  $\Lambda$
for which $f\mid_\Lambda$ is expansive. For every $0\le k \le  r$, the subgroup
 $Z^k(f \mid_{ \Lambda } )$ of $\text{Homeo}(\Lambda)$ is discrete.
In particular, $Z^0(f \mid_{ \Omega(f)})$ is discrete for every $C^r$ Axiom A diffeomorphism and $1\le r\le \infty$. 
\end{lemma}

We observe that taking $f$ restricted to the non-wandering set ${\Omega(f)}$ above is necessary
(see Example~\ref{ex:N-S}) and that one cannot expect triviality of $C^0$-centralizers even for
Anosov diffeomorphisms~\cite{Ro05}.

\subsection{Densely expansive structurally stable diffeomorphisms}

We get back to expansiveness and structural stability.  
We start with the following characterization of structurally stable diffeomorphisms with discrete $C^0$-centralizer.

\begin{maintheorem}\label{abstractthm}
If $f\in \text{Diff}^{\,1}(M)$ is a structurally stable diffeomorphism then the following are equivalent:
\begin{itemize}
\item[(i)] $Z^0(f)$ is discrete;
\item[(ii)] for any $g\in \cC_f$ there exists a unique conjugacy $h_g\in \Homeo$ that is $C^0$-close to the identity;
\item[(iii)] if $\Lambda_{i_s} \succ \dots \succ \Lambda_{i_2} \succ \Lambda_{i_1}$ is a maximal ordered chain of basic
	pieces in $\Omega(f)$ then $\dim W^u(\Lambda_{i_j})= \dim W^u(\Lambda_{i_1})$ is constant for every $1\le j \le s$;
	and 
\item[(iv)] $f$ is an Anosov diffeomorphism.
\end{itemize}
\end{maintheorem}

Ma\~n\'e \cite{Ma75} 
proved that an expansive structurally stable diffeomorphism 
is Anosov. 
The later result will be used to deduce that Ma\~n\'e's result in \cite{Ma75} is optimal: 
there are structurally stable diffeomorphisms with a dense subset of points of expansiveness which are not Anosov 
(cf. Example~\ref{ex:2}).
Theorems~\ref{prop:expansive.points} and ~\ref{abstractthm} together yield the 
following immediate consequence:

\begin{maincorollary}\label{cord}
Assume $f\in \text{Diff}^{\,1}(M)$ is a structurally stable diffeomorphism so 
that the basins of trivial attractors and trivial repellers do not intersect and 
that there exists a maximal totally ordered chain
$\Lambda_{i_s} \succ \dots \succ \Lambda_{i_2} \succ \Lambda_{i_1}$ of basic pieces in $\Omega(f)$ satisfying
$\dim W^u(\Lambda_{i_s})> \dim W^u(\Lambda_{i_1})$. Then 
\begin{itemize}
\item[(a)] $Z^0(f)$ is not discrete,
\item[(b)] $f$ has sensitivity to initial conditions, 
\item[(c)] $f$ admits a dense set of expansiveness. 
\end{itemize}
\end{maincorollary}

The topological description of the foliations and attractors for Axiom A diffeomorphisms is a hard topic not
yet completely understood. Nevertheless, structurally stable diffeomorphisms that satisfy the assumptions
of Corollary~\ref{cord} include: 
(i) structurally stable diffeomorphisms of a compact orientable surface that contains a one-dimensional 
basic set (because these  always admit also periodic sources or sinks (cf.~\cite{Plykin74,Gri})), and 
(ii) structurally stable diffeomorphisms on compact manifolds of dimension larger or equal to $3$ with  
a codimension one orientable expanding attractor (since these have at least one periodic repeller and all remaining 
basic pieces in the spectral decomposition are periodic points as proved in~\cite{GZ}).

\section{Dense expansiveness for structurally stable diffeomorphisms}\label{sec:dense}

This section is devoted to the proof of Theorem~\ref{prop:expansive.points}. Let $f$ be a structurally stable diffeomorphism.
We first recall the dynamics of $f$ on wandering points.

\subsection{Points traveling through filtration elements}\label{travel}
Since $f$ is Axiom A, by the spectral decomposition, the non-wandering set 
is the union of transitive hyperbolic basic pieces: $\Omega(f)=\bigcup_{i=1}^k \Lambda_i$.
We denote by $\cA(f)\subset \Omega(f)$ the set of attractors for $f$.
The strong transversality condition implies that $W^s(x)$ is transversal to $W^u(y)$ for every 
$x,y \in \Omega(f)$ which implies, if these have non-empty intersection, that $\dim E^s_x+ \dim E^u_y \ge \dim M$.
By the no cycles condition, hence existence of a filtration adapted to $\Omega(f)$, 
any maximal totally ordered chain  $\Lambda_{i_s} \succ \dots \succ \Lambda_{i_2} \succ \Lambda_{i_1}$
has at most $k$ elements, the basic set $\Lambda_{i_s}$ is a repeller, the set $\Lambda_{i_1}$ is an attractor, 
and the remaining elements $\Lambda_{i_j}$ are of saddle type (see e.g.\cite{Shub}). 

\begin{remark}\label{rmk:perconn}
For any chain $\Lambda_{i_s} \succ \dots \succ \Lambda_{i_2} \succ \Lambda_{i_1}$  there are points $x_{i_{j}} \in \Lambda_{i_j}$ so that $W^u(x_{i_{j+1}}) \pitchfork W^s(x_{i_{j}}) \neq \emptyset $ for all $j=1\dots s-1$. Since periodic points are dense in $\Omega(f)$, invariant manifolds vary continuously on compact parts and transversality in compact parts is an open condition, we may assume the points $x_{i_j}$ to be periodic.
\end{remark}

The following instrumental lemma allow us to describe the set of points that accumulate, by backward orbits, on the basic sets of
saddle type.  First we need to introduce some notation. Given an attractor $\Lambda$, let $\mathfrak C(\Lambda)$ denote the set of all
maximal and totally ordered chains $\Lambda_{i_s} \succ \dots \succ \Lambda_{i_2} \succ \Lambda_{i_1}$ of basic pieces 
in the non-wandering set $\Omega(f)$ such that $\Lambda_{i_1}=\Lambda$. Observe that the set $\Lambda_{i_s}$ is a repeller.

\begin{lemma}\label{le:non-connections}
Let $f\in \text{Diff}^{\,1}(M)$ be a structurally stable diffeomorphism and let $\Lambda\subset \Omega(f)$ be an attractor for $f$. 
Then for any periodic point $p\in \Lambda$, 
\begin{equation}\label{eq:D}
W^s(p) \setminus \Big[ \bigcup_{\Lambda_{i_s} \succ \dots \succ \Lambda_{i_2} \succ \Lambda_{i_1} \in \mathfrak C(\Lambda)} 
		\; \bigcup_{2\le j \le s-1} W^u(\Lambda_{i_j}) \Big]
\end{equation}
forms an $f$-invariant open and dense subset of $W^s(p)$.
In particular, there exists an open and dense subset of points in $W^s(p)$ whose backward orbit accumulates in some
of the repellers heteroclinically related to $\Lambda$. 
\end{lemma}

\begin{proof}
Since $f\in \text{Diff}^{\,1}(M)$ is structurally stable, the basic pieces in $\Omega(f)$ are localy maximal and 
\begin{equation}\label{eq:M}
M=\bigcup_{x \in \Omega(f)} W^s(x) \; = \; \bigcup_{x \in \Omega(f)} W^u(x).
\end{equation}
Let $\Lambda\subset \Omega(f)$ be an attractor, $p$ be a periodic point in $\Lambda$ and let $D$ be the set defined
by \eqref{eq:D}. 
We claim that $D$ is dense in $W^s(p)$.  Given $x\in W^s(p)$, by \eqref{eq:M} and the strong transversality condition,  
there exists $p_x\in \Omega(f)$ so that $x\in W^u(p_x) \pitchfork W^s(p)$ (see Figure~\ref{fig:heteroclinic} below). 

\begin{figure}[htb]
 \begin{center}
  \includegraphics[width=6cm]{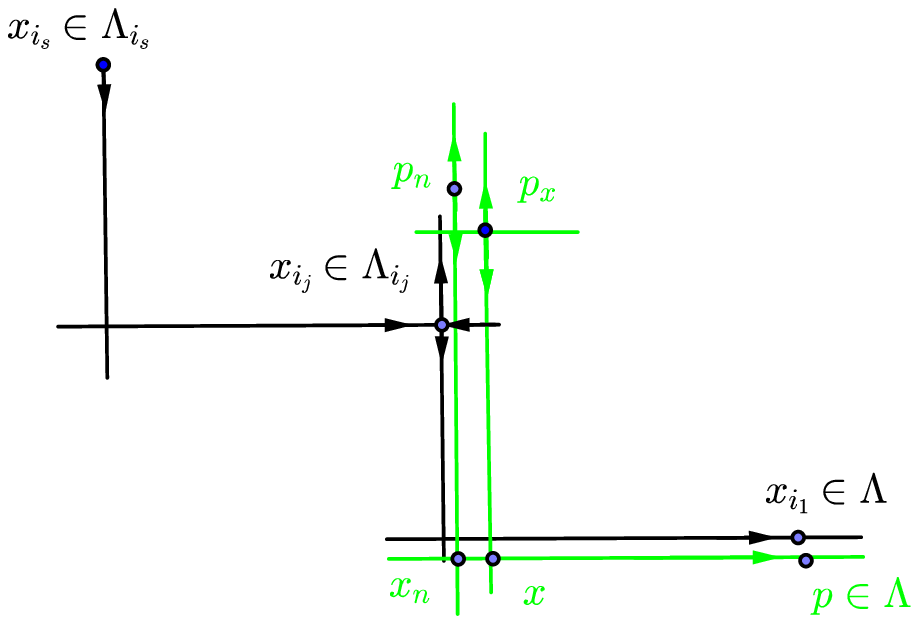}
    \includegraphics[width=7cm]{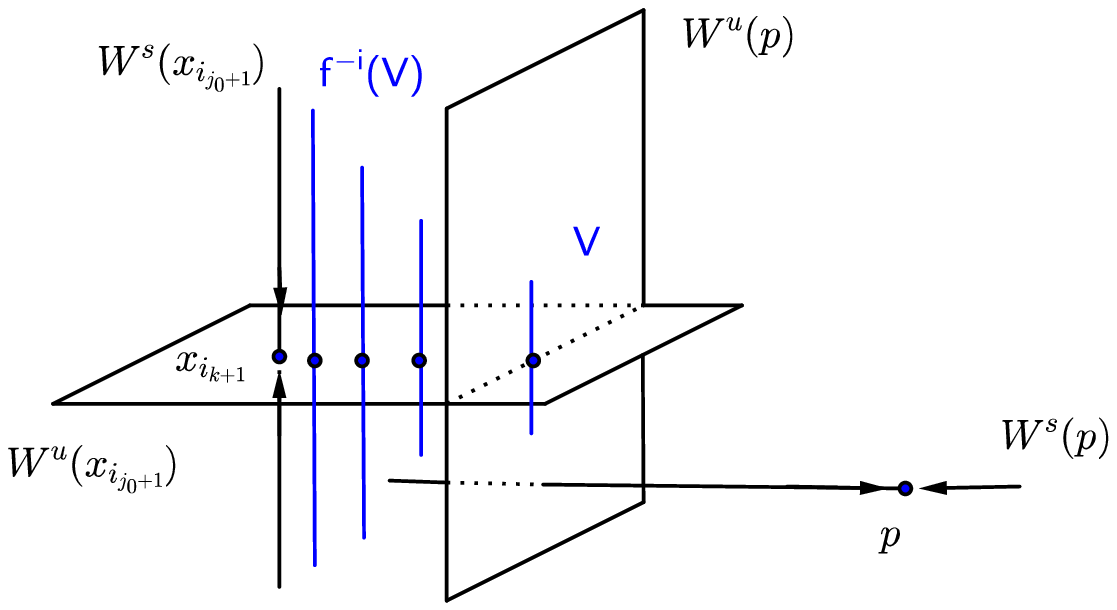}
 \end{center}
 \caption{Heteroclinic intersections associated to the basic pieces (on the left) 
 and selection of a disk $V \subset W^u(x_{i_{k}})$ accumulating on $W^s(x_{i_{k+1}})$ by the $\lambda$-lemma (on the right).}
 \label{fig:heteroclinic}
\end{figure}

Let $\Lambda_{i_s} \succ \dots \succ \Lambda_{i_2} \succ \Lambda_{i_1} \in \mathfrak C(\Lambda)$, $2\le j \le s$ be a maximal chain 
containing the basic set $\Lambda_{i_{j_0}}$ that contains $p_x$. 
If ${j_0}=s$ then $d(f^{-k}(x), \Lambda_{i_s}) \to 0$ as $k\to\infty$, and so $x\in D$ and we are done.
Otherwise $2\le {j_0} <s$ and, by the denseness of periodic points in the non-wandering set and continuity of compact parts of stable and unstable manifolds,  one can take periodic points $p_n \in \Lambda_{i_{j_0}}$ so that
$p_n \to p_x$, and $x$ is approximated by heteroclinic points $x_n \in W^u(p_n) \pitchfork W^s(p)$.
By Remark~\ref{rmk:perconn} there are periodic points $x_{i_{k}} \in \Lambda_{i_k}$ so that 
$W^u(x_{i_{k+1}}) \pitchfork W^s(x_{i_{k}})$ for all $k={j_0}\dots s-1$.
By the $\lambda$-lemma (see e.g. \cite{Shub}), there exists a disk 
$V\subset W^s(p)$ of dimension equal to the stable index of $\Lambda_{i_{{j_0}+1}}$ 
whose iterates by $f^{-1}$ accumulate  (in the $C^1$-topology) on a compact part of 
$W^s(x_{i_{{j_0}+1}})$ (see Figure~\ref{fig:heteroclinic}).
A recursive argument assures that $x$ is accumulated by heteroclinic intersections between the stable manifold of $p$ 
and unstable manifolds of periodic points in $\Lambda_{i_s}$. Since these homoclinic intersections belong to $D$, 
this proves that $D$ is dense in $W^s(p)$.

We are left to prove that $D$ is open in $W^s(p)$. 
 Given $x\in D$ arbitrary, let $n_0\ge 1$ be so that $f^{-n_0}(x)$ belongs to the immediate topological basin
 the repeller. Since the later is an open set, by continuity of $f^{-1}$, there exists an open neighborhood $V$ of $x$ 
 so that $f^{-n_0}(V)$ is also contained in the topological basin of the repeller. 
 This implies that $D$ is an open subset of $W^s(p)$ and finishes the proof of the lemma.
\end{proof}

\begin{remark}\label{rrmk:epellers}
This result asserts that `most' wandering points  travel along the elements of the filtration and converge 
by negative iteration to some of the repellers. Clearly, a similar statement holds for the repellers of $f$ just by 
considering the previous lemma for the diffeomorphism $f^{-1}$.
\end{remark}

\subsection{Proof of Theorem~\ref{prop:expansive.points}} 

Let $f\in \text{Diff}^{\,1}(M)$ be a $C^1$-structurally stable diffeomorphism.

\smallskip 
\noindent \emph{First part: (1) $\Rightarrow$ (2).}
Assume that the topological basins of trivial attractors and repellers in $\Omega(f)$ do not intersect. 
Let
$\Omega(f)=\overline{Per(f)}=\bigcup_{i=1}^k \Lambda_i$ be the spectral decomposition for $\Omega(f)$, 
where each $\Lambda_i$ is the homoclinic class associated to some periodic point 
$p_{i} \in \text{Per}(f) \cap \Lambda_i$. 
Since all periodic points of $\Lambda_i$ are heteroclinically related then: 
\begin{itemize}
\item[(i)] $\Lambda_i=H(p_i)$ (hence $W^s(\mathcal O (p_i))$ is dense in $\Lambda_i$) for every $i=1\dots k$; 
\item[(ii)] $W^s(\Lambda_i):=\{x\in M : d(f^n(x),\Lambda_i) \to 0 \text{  as  } n \to\infty \}$ is an open subset of $M$ 
for every $\Lambda_i\in \cA(f)$, and
\item[(iii)] $W^s(\mathcal O(p_i))$ is dense in $W^s(\Lambda_i)$.
\end{itemize}
Moreover, if $\vep_i>0$ is given by the expansiveness property for $f\mid_{\Lambda_i}$ then we conclude that 
$f\mid_{\Omega(f)}$ is $\vep$-expansive, for $\vep=\min \{ \vep_i : 1\le i \le k\} >0$. 
We need the following useful selection of periodic points.

\medskip \noindent {\bf Claim~1:} 
\emph{There exists a finite set 
$\Theta = \Theta_a \cup \Theta_r $ of periodic points, 
$$
\Theta_a \subset \bigcup\limits_{\Lambda \in \cA(f)} \Lambda 
\quad\text{and}\quad
\Theta_r \subset \bigcup\limits_{\Lambda \in \cA(f^{-1})} \Lambda,
$$
one of which possibly empty,
such that 
\begin{enumerate} 
\item $\big[\bigcup_{p \in \Theta_a} W^s(p)\big] \cup \big[\bigcup_{p\in \Theta_r} W^u(p)\big]$ is dense in $M$,
\item $int(\overline{W^s(\mathcal O(p))}) \cap int(\overline{W^u(\mathcal O(q))})=\emptyset$ for every ${p \in \Theta_a}$ and ${q \in \Theta_r} $,
\item if $\Theta_a\neq \emptyset$ then: 
	\begin{enumerate}
	\item for $p\in \Theta_a$, $W^s(\mathcal O(p))$ is dense in $W^s(H(p,f))$ and its closure 
	contains an 
	open set in the basin of some non-trivial repeller, 
	\item the union of topological  basins of the non-trivial repellers contains an open and dense subset of the 
	stable manifolds $\bigcup_{p\in \Theta_a} W^s(p)$, 
	\end{enumerate}
\item if $\Theta_r\neq \emptyset$ then:
	\begin{enumerate}
	\item for $p\in \Theta_r$, $W^u(\mathcal O(p))$ is dense in $W^u(H(p,f))$ and its closure  
	contains an open set in the basin of some non-trivial 
	attractor,
	\item the union of the topological basin of non-trivial attractors contains an open and dense subset of the unstable
	 manifolds  $\bigcup_{p\in \Theta_r} W^u(p)$.
	\end{enumerate}
\end{enumerate}
}
\medskip

\begin{proof}[Proof of Claim~1]
Let $\cA_*(f) \subset \cA(f)$ denote the set of non-trivial attractors of $f$. Since $f$ is structurally
stable (in particular Axiom A and satisfies the no-cycles condition) then it is well known that
the union of the topological basins of the attractors (resp. repellers) is dense in $M$, that is,
\begin{equation}\label{eq:As}
\overline{\bigcup_{\Lambda \in \cA(f)} W^s(\Lambda)} = \overline{\bigcup_{\Lambda \in \cA(f^{-1})} W^u(\Lambda)} = M.
\end{equation}
Since the basins of trivial attractors and repellers do not intersect, the union of the basins of non-trivial attractors and non-trivial repellers is open and dense in M.
 In other words,
\begin{equation}\label{eq:Astar}
\bigcup_{\Lambda \in \cA_*(f)}  W^s(\Lambda) \; \cup \; \bigcup_{\Lambda \in \cA_*(f^{-1})}  W^u(\Lambda) \quad \text{ is 
open and dense in} \; M.
\end{equation}

The strategy used in the construction of the sets of periodic orbits $\Theta_a$ and $\Theta_r$ 
is to collect periodic orbits primarily among the set of trivial attractors/repellers (under our
assumptions this guarantees that (3)b and (4)b hold on the topological basin of trivial attractors/repellers)
and then to select further periodic orbits among the non-trivial attractors/repellers in a way that conditions (1) and (2) in the claim are satisfied. For that reason periodic orbits in trivial attractors/repellers will always be selected.

Such a selection of periodic orbits (not necessarily unique) can be done as follows, selecting first some periodic
orbits on the attractors and then selecting some periodic orbits among the repellers. 
Let $\cA(f)=\{\Lambda_1,\dots, \Lambda_\ell\}$ be an enumeration of the attractors in $\Omega(f)$ and 
$\Theta_a^0=\emptyset$. If $\Lambda_1=\{\mathcal O(p_1)\}$ is a trivial attractor take $\Theta_a^1=\{\mathcal O(p_1)\}$. 
Otherwise, $\Lambda_1$ is non-trivial 
and there exists a finite number of repellers $\Lambda\in \cA(f^{-1})$ so that $\Lambda \succ \Lambda_1$. 
If all such repellers $\Lambda$ heteroclinically related to $\Lambda_1$ are non-trivial take 
$\Theta_a^1=\{\mathcal O(p_1)\}$ ($p_1$ is any periodic point in $\Lambda_1$) and, otherwise, set $\Theta_a^1=\emptyset$. Proceeding recursively, for any $2\le j \le \ell-1$
if $\Lambda_j=\{\mathcal O(p_j)\}$ is a trivial attractor then take $\Theta_a^j:=\Theta_a^{j-1} \cup \{\mathcal O(p_j)\}$ and, otherwise, 
take $\Theta_a^j:=\Theta_a^{j-1} \cup \{\mathcal O(p_j)\}$ or $\Theta_a^j:=\Theta_a^{j-1}$ depending if 
all repellers heteroclinically related to $\Lambda_j$ are non-trivial
(and $p_j$ is any periodic point in $\Lambda_j$) or not, 
respectively. Take $\Theta_a=\Theta_a^\ell$.

If $\Theta_a\neq \emptyset$ and $\bigcup_{p \in \Theta_a} W^s(p)$ is dense in $M$ just take $\Theta_r=\emptyset$. Otherwise, let $\{\tilde \Lambda_1,\dots, \tilde \Lambda_s\}$ be an enumeration of the 
repellers in 
$\Omega(f)$ that are not heteroclinically related with any of the attractors that contain the periodic points in $\Theta_a$,  and take 
$\Theta_r=\{\mathcal O(q_1), \dots, \mathcal O(q_s)\}$ be formed by periodic orbits with 
$q_i\in \tilde \Lambda_i$, for $1\le i \le s$. In particular the sets $\Theta_a$ and $\Theta_r$
cannot be simultaneously empty.

Items (1), (3)(a) and (4)(a) in the claim are immediate from \eqref{eq:As} and the construction.
Item (2) follows from the fact that the repellers determining $\Theta_r$ are chosen not heteroclinically related
to any point of $\Theta_a$.
Finally, using ~\eqref{eq:Astar} and Lemma~\ref{le:non-connections} we conclude that items (3)(b) and (4)(b) 
also hold. This proves the claim.
\end{proof}

By construction, the union of topological  basins of the non-trivial repellers contains an open 
and dense subset $O_p \subset W^s(p)$ for any $p\in \Theta_a$ (and a similar statement holds for every $p\in \Theta_r$, if $\Theta_r$ is non-empty).   For every $p\in \Theta_a$,  resp. $\Theta_r$, let $P_p$ be the union of all periodic points in the (non-trivial) repellers, resp. attractors, of $f$ that are heteroclinically related to $H(p,f)$. In other words,
$$
P_p = \bigcup_{\substack{ \Lambda \succ H(p,f)\\ \Lambda \in \cA(f^{-1})}} P_{p,\Lambda} 
$$
where $P_{p,\Lambda} =  \text{Per}(f) \cap \Lambda$.
Now, if $P$ denotes the union of such periodic points taken over all points $p\in \Theta_a \cup \Theta_r$ then, 
since any countable infinite set is a disjoint countable union of countable infinite sets and 
$\sharp (\Theta_a\cup\Theta_r) <\infty$ we claim that there exists a (not necessarily unique) decomposition 
\begin{equation}\label{eq:disjointP}
P=\bigsqcup_{p\in \Theta_a\cup\Theta_r} \bigsqcup_{\ell \in \mathbb Z} P_{p,\ell}
\end{equation}
 as a disjoint union of countable infinite sets $P_{p,\ell} \subset P_p$ in such a way that for
$p\in \Theta_a$ (resp. for $p\in \Theta_r$) the set $P_{p,\ell}$ contains infinitely many 
periodic points in each of the non-trivial repellers whose basin intersects $W^s(\mathcal O(p))$ (resp.
non-trivial attractors whose basin intersects $W^u(\mathcal O(p))$). 
Indeed, given $\Lambda \in \cA(f^{-1})$ such that $\Lambda \succ H(p,f)$, as there are countable 
infinitely many periodic points in $\text{Per}(f) \cap \Lambda$ then there exists a bijection $\iota_\Lambda : \mathbb Z \times \mathbb Z \to
P_{p,\Lambda}$. Then, for any $\ell \in \mathbb Z$, the set 
$$
P_{p,\ell}:=\bigcup_{\Lambda \succ H(p,f), \Lambda \in \cA(f^{-1})} \iota_\Lambda(\{\ell\} \times \mathbb Z)
	\subset P_p
$$ 
satisfies the requirements of the claim 
involving ~\eqref{eq:disjointP}.
 We use the following key claim:

\medskip \noindent {\bf Claim~2:} \emph{For every periodic point $p\in \Theta_a$ there exists a countable and dense subset $D_p \subset 
W^s(\mathcal O(p))$ 
so that the following holds: for any $x,y\in D_p$ there are periodic points $p_x \neq p_y$ in ${P_p=}\bigsqcup_{\ell\in \mathbb Z}
P_{p, \ell}$ such that $x\in W^u(p_x)$ and $y\in W^u(p_y)$. In particular, there exists 
$n=n_{x,y} \ge 1$ for which $d(f^{-n}(x), f^{-n}(y))>\frac\vep{2}$.} 

\medskip

Since claim also holds for the periodic points in $\Theta_r$ (replacing $f$ by $f^{-1}$, in which case
unstable manifolds for $f$ become stable manifolds for $f^{-1}$), by the decomposition of periodic points in 
\eqref{eq:disjointP} we conclude that  there exists a countable and dense subset of $\bigcup_{p \in \Theta_a} 
W^s(\mathcal O(p)) 
\cup \bigcup_{p\in \Theta_r} W^u(\mathcal O(p))$  formed by points that are $\frac\vep{2}$ separated by either positive or negative 
iterations of $f$. By item (1) above, the later set is dense in $M$.
Thus, in order to complete the proof of the first part of the theorem we are left to prove the claim.

\begin{proof}[Proof of Claim~2]
Fix $p\in \Theta_a$ and let $\Delta \subset {W^s(\mathcal O(p))}$ be a fundamental domain of 
{$W^s(\mathcal O(p))$}. 
Given $\ell\in \mathbb Z$ define $\Delta_\ell = f^\ell(\Delta)$. Observe that 
${W^s(\mathcal O(p)) \setminus\{\mathcal O(p)\} } =\bigsqcup_{\ell\in \mathbb Z}  \Delta_\ell$.
Fix $\ell \in \mathbb Z$ and set $P^1_\ell:=P_{p,\ell} \, {\subset P_p}$.  Given any open covering of $\overline{\Delta_{\ell}}$ by 
open balls of radius $1/2$,  the compactness of the closure $\overline{\Delta_\ell}$ guarantees one can extract 
a finite covering $\{B^2_i\}_i$. 
The $\lambda$-lemma (as used in the proof of Lemma~\ref{le:non-connections}) implies all points in $\overline{\Delta_\ell}$ are accumulated 
by heteroclinic intersections between the unstable manifolds of every periodic point in $P^1_\ell$ and the stable manifold of $p$. 
Thus for every $i$ there exists $p^2_i \in P_\ell^1$ so that the intersection $W^u(p^2_i) \pitchfork W^s(p)$ contains some point $x_{i, 2}$ in $B^2_i$. Since periodic points in the same homoclinic class are homoclinically related, we can choose the periodic points $\{ p_{i,2}\}_i$
to be distinct. Moreover, by construction, the set $\{ x_{i,2}\}_i$ (we omit the dependence on $\ell$ for notational simplicity) 
is finite and $1/2$-dense in $\overline{\Delta}_\ell$. Moreover the set  $P_{\ell}^{2}:= P^1_{ \ell} \setminus \cup_i \{ p^2_i\}$
has still infinitely many periodic points in each of the repellers 
heteroclinically related to the attractor $\Lambda$ that contains $p$.

\begin{figure}[htb]
\includegraphics[height=6cm, width=9cm]{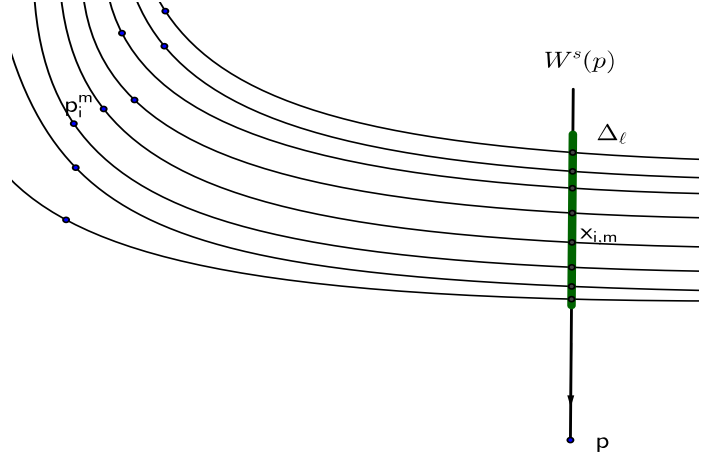}
\caption{Representation of $\frac1m$-dense heteroclinic intersections $\{x_{i,m}\}$ associated to periodic points in some 
non-trivial repeller in the fundamental domain $\Delta_\ell$ of $W^s(p)$.}
\end{figure}

Proceeding recursively, for every  $m\ge 1$ we obtain a finite number of periodic points $\{ p_i^m \} \in P_\ell^m$
whose heteroclinic intersections of the corresponding unstable manifolds with $W^s(p)$ contains a $1/m$-dense 
set $\{ x_{i,m}\}_i$ of points in $\overline{\Delta}_\ell$,
and the set  $P_{\ell}^{m+1}:= P^m_{ \ell} \setminus \cup_i \{ p^m_i\}$  has still infinitely many periodic points in each 
of the repellers heteroclinically related to $\Lambda$.
The resulting set $D_{p,\ell}:= \bigcup_{m\ge 1} \{x_{i,m}\}$, formed by points obtained by the heteroclinic intersections, 
is a countable dense subset in $\overline{\Delta_\ell}$ and the set $D_{p}:= \bigcup_{\ell\ge 1} D_{p,\ell}$ satisfies the 
requirements of the lemma: for any $x,y\in D_p$
there are periodic points $p_x \neq  p_y \in \text{Per}(f)$ such that $x\in W^u(p_x)$ and $y\in W^u(p_y)$. In consequence, 
$$
d(f^{-n}(x), f^{-n}(y))
	\ge d(f^{-n}(p_x), f^{-n}(p_y)) - d(f^{-n}(p_x), f^{-n}(x)) - d(f^{-n}(p_y), f^{-n}(y))
$$
can be taken larger that $\vep/2$ provided that $n$ is large enough (here we used that the set $\text{Per}(f)$ is negatively
expansive by $f$). This completes the proof of the lemma.
\end{proof}

\smallskip
\noindent \emph{Second part: (2) $\Rightarrow$ (3).} This is immediate.

\smallskip
\noindent \emph{Third part: (3) $\Rightarrow$ (1).} We prove this by contradiction. Assume that there are periodic
points $p_1,p_2$ so that  $\Lambda_1=\mathcal O(p_1)$ is a trivial repeller, $\Lambda_2=\mathcal O(p_2)$
is a trivial attractor and $W^u(\Lambda_1) \cap W^s(\Lambda_2) \neq \emptyset$. 
By uniform continuity, $f$ is sensitive to initial conditions if and only if $f^k$ ($k\ge 1$) is sensitive to initial conditions  
(although the constants of separability could change).
Hence, replacing $f$ by some suitable iterate $f^k$, we may assume without loss of generality that both $p_1,p_2$ 
are fixed points for $f$. The existence of a filtration 
guarantees that $W^u(\Lambda_1) \cap W^s(\Lambda_2)$ contains a non-empty open set $V$.
We claim that for any $\vep>0$ there exists an open subset $V_\vep\subset V$ of points so that $d(f^n(x), f^n(y))<\vep$
for all $n\in \mathbb Z$. Indeed, if $k=k(\vep,V)\ge 1$ is so that $f^j(V)\subset B(p_2,\vep)$ and
$f^{-j}(V)\subset B(p_1,\vep)$ for every $j\ge k$, and $x\in V$ it is enough to take the open set
$$
V_\vep=\{ y\in V : d( f^j(x), f^j(y) ) <\vep, \; \forall |j| \le k \}.
$$
This proves that there exists an open set of points whose iterates always remain $\vep$-close. 
Since $\vep$ was chosen arbitrary this contradicts the sensitivity to initial conditions assumption on $f$.  
The proof of the theorem is now complete.

\begin{example}\label{ex:choiceT2}
Assume that $f: \mathbb T^2 \to \mathbb T^2$ is the standard derived from Anosov diffeomorphism (see e.g. \cite[pp 539]{KH}). It is well known that $f$ is Axiom A and that $\Omega(f)=\{p\} \cup \Lambda$, where $p$ is a repeller and  $\Lambda$ is a non-trivial attractor. In this case it is trivial that $\Theta_r=\emptyset$ and $\Theta_a=\{p\}$ 
satisfy the requirements (1)-(3) in the first part of the proof of Theorem~\ref{prop:expansive.points}.
\end{example}

\begin{example}\label{ex:combinatorics}
Assume that $f: M\to M$ is an Axiom A diffeomorphism such that $\Omega(f)=\bigcup_{j=1}^5 \Lambda_j$, where 
$(\Lambda_j)_{j=1, 2}$ are attractors and $(\Lambda_j)_{j=3,4,5}$ are repellers. Let $p_j$ denote a periodic point in $\Lambda_j$, for every $j\in\{1,2,\dots, 5\}$. Assume further that 
$$
\Lambda_3 \succ \Lambda_1
\qquad
\Lambda_4 \succ \Lambda_1
\qquad
\text{and}
\quad
\Lambda_5 \succ \Lambda_2
$$
and that $\Lambda_1$ consists of a periodic sink. In this case, the sets $\Theta_a=\{\mathcal O(p_1), \mathcal O(p_2)\}$ and 
$\Theta_r=\emptyset$ (chosen according to the selection process in the proof of the theorem) satisfy conditions
(1)-(3). The diffeomorphism $f^{-1}$ is also Axiom A and, in this case, we get $\Theta_a=\{\mathcal O(p_5)\}$ and
$\Theta_r=\{\mathcal O(p_1)\}$.
\end{example}

\section{$C^0$-Conjugacy classes and $C^0$-centralizers of homeomorphisms}\label{sec:proofs}

\subsection{Proof of Theorem~\ref{prop:bijection}}
Fix $r\ge 0$ and $f \in \text{Diff}^{\,r}(M)$. For any $g\in \cC_f$ and 
$h\in \mathcal H_{f,g}$, consider the map 
$$
\begin{array}{rccc}
F_h \colon &  Z^0(f) & \to & \mathcal H_{f,g} \\ \medskip
	& \tilde f & \mapsto & h \circ \tilde f,
\end{array}
$$
where the subsets $Z^0(f), \mathcal H_{f,g}$ of $\Homeo$ are endowed with the distance $d_0$.
We claim that $F_h$ is well defined. Indeed, for any $ \tilde f \in Z^0(f)$ it holds $\tilde f \circ f = f\circ \tilde f $. Thus
$
(h\circ \tilde f )\circ f  
	 = h\circ (\tilde f\circ f)  = (h\circ  f )\circ \tilde f 
	 = (g\circ h) \circ \tilde f = g\circ (h \circ \tilde f).
$
This can be also observed by the commutative diagram:
\begin{center}
\begin{tikzpicture}
  \matrix (m) [matrix of math nodes,row sep=3em,column sep=4em,minimum width=2em]
  {
     M & M & M \\
     M & M & M \\};
  \path[-stealth]
    (m-1-1) edge node [left] {$f$} (m-2-1)
            edge node [above] {$\tilde f$} (m-1-2)
    (m-2-1.east|-m-2-2) edge node [below] {$\tilde f$}
    (m-2-2)
    (m-1-2) edge node [right] {$f$} (m-2-2) 
    edge node [above] {$h$} (m-1-3)
     (m-2-2) edge node [below] {$h$} (m-2-3)
     (m-1-3) edge node [right] {$g$} (m-2-3);
\end{tikzpicture}
\end{center}
This proves that $h\circ \tilde f$ is also a $C^0$-conjugacy between $f$ and $g$. 
We proceed to prove that $F_h$ is indeed a homeomorphism.
Since $h$ is a homeomorphism then $F_h$ is clearly injective. To prove the continuity of $F_h$, fix an arbitrary 
$\tilde f \in Z^0(f)$. By the uniform continuity of $h$, given $\vep>0$ there exists $0<\delta <\vep$ so that 
$d(h(x),h(y))<\vep$ for any points $x,y\in M$ satisfying $d(x,y)<\delta$. 
In particular, if $\tilde f, \hat f \in Z^0(f)$ and $d_0(\tilde f, \hat f)<\delta$ then
\begin{align*}
d_0(F_h(\tilde f), F_h(\hat f)) & = \max\{ \sup_{x\in M} d(h ( \tilde f(x)), h( \hat f(x)) ), \sup_{x\in M} d(\tilde f^{-1} ( h^{-1}
(x)), \hat f^{-1} ( h^{-1}(x) )) \} \\
& < \max\{ \vep , d_{0}(\tilde f, \hat f)\} =\vep,
\end{align*}
which proves the continuity of $F_h$.
To prove that $F_h$ is surjective, given $\tilde h \in \mathcal H_{f,g}$ write $\tilde h=h \circ (h^{-1}\circ \tilde h)$ and note that $h^{-1}\circ \tilde h$ is a homeomorphism. Moreover 
$$
f\circ (h^{-1}\circ \tilde h)
	= (f\circ h^{-1})\circ \tilde h
	= (h^{-1}\circ g) \circ \tilde h 
	= h^{-1}\circ (g \circ \tilde h)
	= (h^{-1} \circ \tilde h) \circ f,
$$
which can also be read from the commutative diagram
\begin{center}
\begin{tikzpicture}
  \matrix (m) [matrix of math nodes,row sep=3em,column sep=4em,minimum width=2em]
  {
     M & M & M \\
     M & M & M \\};
  \path[-stealth]
    (m-1-1) edge node [left] {$f$} (m-2-1)
            edge node [above] {$\tilde h$} (m-1-2)
    (m-2-1.east|-m-2-2) edge node [below] {$\tilde h$}
    (m-2-2)
    (m-1-2) edge node [right] {$g$} (m-2-2) 
    edge node [above] {$h^{-1}$} (m-1-3)
     (m-2-2) edge node [below] {$h^{-1}$} (m-2-3)
     (m-1-3) edge node [right] {$f$} (m-2-3);
\end{tikzpicture}
\end{center}
This proves that $h^{-1}\circ \tilde h \in Z^0(f)$, and so
$F_h$ is a continuous bijection whose inverse map is
$F_h^{-1} \colon  \mathcal H_{f,g}   \to  Z^0(f)$ given by 
$F_h^{-1}(\tilde h) =h^{-1} \circ \tilde h$.
Since $h^{-1}$ is continuous, hence uniformly continuous, then $F_h^{-1}$ is clearly continuous. Altogether this proves that  $F_h$ is a homeomorphism and finishes the proof of the first part of the proposition. 

A simple computation shows that $F_h( \tilde f \circ \hat f) = h\circ (\tilde f \circ \hat f) = (h\circ \tilde f) \circ h^{-1} \circ (h \circ \hat f) = F_h(\tilde f) \circ h^{-1} \circ F_h(\hat f)$ for every $\tilde f, \hat f \in Z^0(f)$. 
Property (i) is immediate. Property (iii) is a direct consequence of the fact that homeomorphisms preserve discrete sets. 
We are left to prove property (ii). Since the elements in the decomposition $\cH_f = \bigcup_{g\in \cC_f} \cH_{f,g}$ 
are pairwise disjoint, for any $h\in \cH_f$ there exists a unique $g\in \cC_f$ so that $h\in \cH_{f,g}$ (we write for simplicity
$h=h_g$). The map
$$
\begin{array}{rccc}
F \colon &  \cH_f & \to &  \cC_f \\ \medskip
	& h=h_g & \mapsto & g
\end{array}
$$
is continuous. Indeed, given $\vep>0$ take $\delta=\vep/2$ and
assume that $d_0(h, \tilde h)<\delta$, where $h=h_{g}$ and $\tilde h= \tilde h_{\tilde g}$.
Observe that $ g \circ (h_g \circ \tilde h_{\tilde g}^{-1}) = (h_g \circ \tilde h_{\tilde g}^{-1}) \circ \tilde g$ 
(in other words $\bar h := h_g \circ \tilde h_{\tilde g}^{-1} \in \cH_{\tilde g, g}$). Moreover,
$
d(\bar h(x), x) 
	= d(h_g (\tilde h_{\tilde g}^{-1} x), \tilde h_{\tilde g} (\tilde h_{\tilde g}^{-1}x)) < \delta
$
for every $x\in M$. This, together with identical computations for $\bar h^{-1}$ implies that
$
d_0( \bar h, id) <\delta. 
$
Thus, 
$$
d_0( g, \tilde g) =  d_0( \bar h \circ \tilde g \circ \bar h^{-1}, \tilde g )
		\le d_0( \bar h \circ \tilde g \circ \bar h^{-1}, \tilde g \circ \bar h^{-1})
		+ d_0(\tilde g \circ \bar h^{-1}, \tilde g ) <\vep.
$$
This proves the continuity of $F$. Since $F^{-1}(\{g\})=\cH_{f,g} \simeq Z^0(f)$ for every $g\in \cC_f$ then 
we conclude that $\cH_f$ is homeomorphic to $\cC_f \times Z^0(f)$.
This proves (ii) and completes the proof of the proposition. \hfill $\square$

\medskip

Although the argument used in the proof of Lemma~\ref{thm:C0-homoclinic} is essentially contained in~\cite{Wa70}, we include it here for completeness.

\subsection{Proof of Lemma~\ref{thm:C0-homoclinic}}

Take  $r\ge 0$ and  assume $f \in \text{Diff}^{\,r}(M)$ has an $f$-invariant subset  $\Lambda$
such that $f\mid_\Lambda$ is expansive. Fix $0\le k \le  r$. By the inclusion 
$Z^0(f \mid_{ \Lambda } ) \supset Z^k(f \mid_{ \Lambda } )$ it is enough to prove that 
$Z^0(f\mid_{\Lambda})$ is a discrete subgroup of $\text{Homeo}(\Lambda)$.
Due to the subgroup structure in $\Homeo$, in order to prove that $Z^0(f\mid_{{\Lambda}})$ is
discrete it is enough to prove that there exists $\delta>0$ so that any 
$h\in Z^0(f\mid_{ \Lambda})$ with $d_0(h,id_{\Lambda})<\delta$ coincides with the identity map $id_{{\Lambda}}$.
Let $\vep>0$ be an expansiveness constant for $f\mid_{\Lambda}$ and $\delta=\vep/2$. Take $h\in Z(f\mid_{\Lambda})$ so that 
$d_0(h,id\mid_{\Lambda})<\delta$. Since $\Lambda$ is an $f$-invariant set and
$h\circ f = f\circ h$ on $\Lambda$ then 
$$
d( f^n(h(x)), f^n(x)  )
	= d( h(f^n(x)), f^n(x)  )
	\le d_0( h, id) <\vep/2
$$
for any $x\in \Lambda$ and every $n\in \mathbb Z$. By expansiveness it follows that $h=id\mid_\Lambda$.
To finish the proof of the lemma, just observe that $f$ is an Axiom A diffeomorphism then ${\Omega(f)}$ consists
of a finite number of hyperbolic, hence expansive, homoclinic classes. 
In particular, $Z^0(f \mid_{ \Omega(f)})$ is discrete for every $C^r$ Axiom A diffeomorphism and $1\le r\le \infty$. 
This completes the proof of the lemma.

\begin{remark}
In the case of an Axiom A diffeomorphism $f$ there exists $\vep>0$ and a $C^1$ neighborhood $\cU$ of $f$ 
so that the expansiveness constant $\vep_g$ for $g\mid_{\Omega(g)}$ is uniformly bounded below by $\vep>0$ for all 
$g\in \mathcal U$.
Thus, there exists $\vep>0$ so that $d_0(h,id)>\vep$ for any  $h\in Z^0(g\mid_{H(p_g,g)})$ and $g\in \cU$.
This means that the smallest distance to identity in $Z^0(f\mid_{\Omega(f)}) \setminus \{id\}$ can be taken uniform in a small
neighborhood of $f$.
\end{remark}

\section{Dense expansiveness and conjugacies $C^0$-close to identity}\label{sec:uniqueness}

The purpose of this section is to prove Theorem~\ref{abstractthm} and their consequences.

\subsection{Proof of Theorem~\ref{abstractthm} }

Let $f\in \text{Diff}^{\,1}(M)$ be a structurally stable diffeomorphism, hence Axiom A.
It follows from the spectral decomposition theorem for Axiom A diffeomorphisms that  
$\Omega(f)=\bigcup_{i=1}^k \Lambda_i$ where each $\Lambda_i$ is a hyperbolic homoclinic class
associated to a periodic point $p_{i} \in \text{Per}(f)$.  The equivalence between (i) and (ii) 
follows from Theorem~\ref{prop:bijection}. Here we prove the remaining equivalences as follows:

\medskip
\noindent \emph{$(iii) \Rightarrow (iv)$}
\smallskip

\noindent Assume that for any maximal totally ordered chain $\Lambda_{i_s} \succ \dots \succ \Lambda_{i_2} \succ \Lambda_{i_1}$ of basic pieces in $\Omega(f)$ one has $\dim W^u(\Lambda_{i_j})= \dim W^u(\Lambda_{i_1})$ is 
constant for every $1\le j \le s$. Since $f$ is Axiom A and satisfies the strong transversality condition, 
for any $x\in M$ there exist $y,z\in \Omega(f)$ uniquely determined so that $x\in W^u(y) \pitchfork W^s(z)$.
In particular, $\dim W^u(y)+ \dim W^s(z)=\dim M$ and there exists $\vep_x>0$ so that 
$\mathcal F^u_{\vep_x}(x) \pitchfork \mathcal F^s_{\vep_x}(x) =\{x\}$, where $\mathcal F^u_{\vep_x}(x)$ 
denotes the $\vep_x$-ball around $x$ in $W^u(y)$ and $\mathcal F^s_{\vep_x}(x)$ 
denotes the $\vep_x$-ball around $x$ in $W^s(z)$. 
Since $M$ is compact and the invariant foliations vary continuously with the point on compact parts then there exists
$\vep_0>0$ so that $\mathcal F^u_{\vep_0}(x) \pitchfork \mathcal F^s_{\vep_0}(x) =\{x\}$ for every $x\in M$.
As $\Omega(f)$ is a hyperbolic set, reducing $\vep_0$ if necessary, we may assume that 
for every $x\in \Omega(f)$, the stable set
$$
B^f_\infty(x,\vep_0):=\{y\in M \colon d(f^n(x), f^n(y))<\vep_0 \;\text{for every}\;  n\in \mathbb Z_+\}
$$
is contained in local stable manifold $\cF_{\vep_0}^s(x)$ (cf. stable manifold theorem in \cite{Shub}).
We may reduce $\vep_0$ if necessary so that $\min_{i\neq j} \dist (\Lambda_i,\Lambda_j)>\vep_0$.

We proceed to prove that $f$ is expansive. We claim that $\vep_0$ is an expansiveness constant for 
$f\mid_{\Omega(f)}$.
Thus, if $x,y \in \Omega(f)$ satisfy $d(f^n(x), f^n(y))<\vep_0$ for every $n\in \mathbb Z$ then $y\in \cF_{\vep_0}^s(x)
\pitchfork \cF_{\vep_0}^u(x)$ and, consequently, $y=x$. This proves that $\vep_0$ is an expansiveness constant
for $f\mid_{\Omega(f)}$. 
Fix $0<\vep<\frac{\vep_0}3$.  Now we prove that $B^f_\infty(x,\vep) \subset \cF^s_\vep(x)$ for every $x\in M\setminus \Omega(f)$. 
Assume that $x \in M \setminus \Omega(f)$ and that $y\in M$ is such that $d(f^n(x), f^n(y))<\vep$ for every $n\in \mathbb Z_+$. 
There are $p_x, p_y\in \Omega(f)$ so that $x\in W^s(p_x)$ and $y\in W^s(p_y)$ (cf. ~\eqref{eq:M}). 
Recall $\min_{i\neq j} \dist (\Lambda_i,\Lambda_j)>\vep_0$. By triangular inequality, there exists $N\ge 1$ large so that
$$
d(f^{n}(p_x), f^{n}(p_y)) \le d(f^{n}(x), f^n(y)) + d(f^{n}(y), f^n(p_y)) + d(f^{n}(x), f^n(p_x))
	\le \vep+ 2 C\lambda^n \vep
<\vep_0
$$
for every $n\ge N$. This proves that $f^N(p_y) \in B^f_\infty(f^N(p_x),\vep) \subset \cF^s_\vep(f^N(p_x)) \subset W^s(f^N(p_x))$
and, consequently, $p_y\in W^s(p_x)$. This proves that $y\in \cF_\vep^s(x)$, as desired.
Finally, if $x,y\in M$ and $d(f^n(x), f^n(y))<\vep$ for every $n\in \mathbb Z$ then 
$
y\in B^f_\infty(x,\vep) \cap B^{f^{-1}}_\infty(x,\vep) 
	\subset \cF^s_\vep(x) \cap \cF^u_\vep(x) = \{x\}.
$
This proves that $f$ is $\vep$-expansive. By \cite{Ma75}, $f$ is an Anosov diffeomorphism.

\medskip
\noindent \emph{$(iv) \Rightarrow (i)$}
\smallskip

\noindent Since every Anosov diffeomorphism is expansive, (i) follows from Lemma~\ref{thm:C0-homoclinic}.

\medskip
\noindent \emph{$(i) \Rightarrow (iii)$}
\smallskip

\noindent If $f$ is an Anosov diffeomorphism then both properties (i) and (iii) hold. For that reason we need only to
consider $C^1$-structurally stable diffeomorphisms that are not Anosov.
By structural stability, for any maximal totally ordered chain 
$\tilde \Lambda_{i_s} \succ \dots \succ \tilde \Lambda_{i_2} \succ \tilde \Lambda_{i_1}$ of basic pieces in $\Omega(f)$ 
recall that $\dim W^u(\tilde \Lambda_{i_j}) \le \dim W^u(\tilde \Lambda_{i_{j+1}})$ for every $1\le j \le s-1$.
We prove this implication by contraposition.  Assume that there exists a maximal totally ordered chain 
$\Lambda_{i_s} \succ \dots \succ \Lambda_{i_2} \succ \Lambda_{i_1}$ of basic pieces in $\Omega(f)$ 
so that $\dim W^u(\Lambda_{i_1}) < \dim W^u(\Lambda_{i_s})$. Our purpose is to prove that $Z^0(f)$ is not discrete.

Since $\Omega(f)$ is a hyperbolic set for $f$ (hence $f\mid_{\Omega(f)}$ is expansive) 
there exists $\vep_0>0$ an expansiveness constant for $f\mid_{\Omega(f)}$.
Given any element $h\in Z^0(f)$ it preserves the non-wandering set of $f$, that is $h(\Omega(f))=\Omega(f)$.
Therefore, if in addition it satisfies $d_{C^0}(h,id)<\vep_0$ then $h\mid_{\Omega(f)}=id$ (cf. Lemma~\ref{thm:C0-homoclinic}).
Furthermore, since any element in $Z^0(f)$ preserves stable and unstable foliations then
$h(W^s(x))=W^s(x)$ and $h(W^u(x))=W^u(x)$ for every $x\in \Omega(f)$.
In particular $h$ preserves heteroclinic points. 

The remaining of the proof is constructive and we build a continuum of homeomorphisms $C^0$-close to the identity 
that belong to $Z^0(f)$. 
By Lemma~\ref{le:non-connections}, there exists a periodic point in $\Lambda_{i_1}$ that is heteroclinically related
to some periodic point in $\Lambda_{i_s}$. Moreover, since compact parts of stable and unstable manifolds vary 
continuously with the point and there exists a filtration, the argument used in the final part of the proof of 
Lemma~\ref{le:non-connections} guarantees that there exists an open set 
$U \subset M \setminus \Omega(f)$ so that $\dist_H(f^n(U), \Lambda_{i_1}) \to 0$ and 
$\dist_H(f^{-n}(U),\Lambda_{i_s}) \to 0$ as $n$ tends to infinity
(here $\dist_H$ denotes the Hausdorff distance).
Since the sets $\Lambda_{i_1}$ and $\Lambda_{i_s}$
are locally maximal then $W^s(\Lambda_{i_1})= \bigcup_{x\in \Lambda_{i_1}} W^s(x)$ and 
$W^u(\Lambda_{i_s})= \bigcup_{x\in \Lambda_{i_s}} W^u(x)$. In particular every point in $U$ belongs to a stable 
manifold of some point of $\Lambda_{i_1}$ and to an unstable manifold of some point of $\Lambda_{i_s}$. 
Diminishing $U$ if necessary, we may assume that 
the collection of compact sets $(\overline{f^n(U)})_{n\in \mathbb N}$ is pairwise disjoint, and that $U$ is foliated by pieces of
stable and unstable disks.  

Since we assumed $\dim W^u(\Lambda_{i_1}) < \dim W^u(\Lambda_{i_s})$ then $\dim W^u(\Lambda_{i_s}) 
+ \dim W^s(\Lambda_{i_1}) > \dim M$ for every $x\in U$. 
Let $\alpha_x \subset U$ be a smooth submanifold of dimension one contained 
in the intersection $\mathcal F^u(x) \cap \mathcal F^s(x)$ (see Figure~3 below). 
\begin{figure}[htb]
\includegraphics[height=4cm, width=9cm]{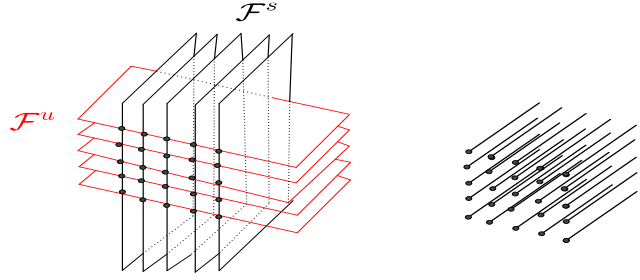}
\caption{Representation of the intersection of pieces of stable and unstable foliations in $U$ (on the left), and smooth submanifolds contained in their intersection (on the right).}
\end{figure}
Up to take a smaller $U$ if necessary, by continuity of the  intersections between invariant manifolds one may choose the 
family $(\alpha_x)_{x\in U}$ to vary continuously with $x$ and to partition $U$. By convention, $\alpha_x=\alpha_z$ for every $z\in \alpha_x$.
Up to conjugacy by a homeomorphism on $U$, there exist natural coordinate system given by arc length. More precisely,
if $\ell(\alpha_x)$ denotes the length of the one-dimensional submanifold $\alpha_x$
and $E \subset W^u(\Lambda_{i_s}) \cap U$ is the continuous submanifold that intersects each submanifold $\alpha_x$
exactly on its midpoint, the parametrization of $\alpha_x$ by arc length with the same orientation for all
points (which we denote by $\alpha_x: ]-\ell(\alpha_x)/2,  \ell(\alpha_x)/2[ \to M$ by some abuse of notation) induces a natural homeomorphism 
$$
\begin{array}{rcc}
\varphi: U & \to & \tilde U:= \{ (x, \ell) \in E \times \mathbb R \colon  x\in E, \;  -\ell(\alpha_x)/2 < \ell < \ell(\alpha_x)/2 \} \\
	z & \mapsto & (x,\ell)
\end{array}
$$ 
where $\alpha_z \cap E =\{x\}$ and $-\ell(\alpha_x)/2 < \ell < \ell(\alpha_x)/2$ is unique so that $\alpha_x(\ell)=z$.
Now, let $V\subsetneq W \subsetneq U$ be two small open sets so that $V\subsetneq \overline V \subsetneq W \subsetneq \overline W \subsetneq U$ and $\ell(\alpha_x)\ge \ell_0>0$  for all points $x\in V$ (see Figure~4). 
Observe that $\varphi\mid_{\overline W}: \overline W \to \varphi(\overline W)$ and its inverse are 
uniformly continuous.
\begin{figure}[htb]
\includegraphics[scale=.40]{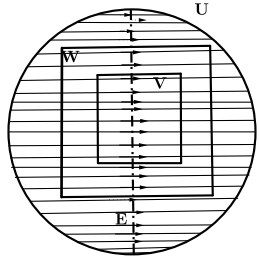}
\caption{Foliated chart}
\end{figure}
Since every element in the centralizer preserves heteroclinic points this motivates a perturbation 
of the identity map along this one-dimensional fibers. Fix a continuous `bump function' $\beta : \mathbb R \to [0,1]$ so that 
$\beta\mid_{\mathbb R\setminus [- \ell_0/2,\ell_0/2]} \equiv 0$  and $\beta\mid_{[- \ell_0/4,\ell_0/4]} \equiv 1$. Fix $\zeta>0$ small and $t\in [0,1]$. For any $z\in U$, consider the one-parameter family
$$
h_{0,t}(z) :=
	\begin{cases} 
	\begin{array}{cl}
	 \varphi^{-1} \circ H_t \circ \varphi(z) &, \text{if}\, z\in W \\  \vspace{.2cm}
	 z &, \text{if} \, z\in U\setminus W
	\end{array}
	\end{cases}
$$
where $H_t: W \to W$ is an 
given by $(x,\ell) \mapsto (x, \ell + \beta(\ell) t\zeta)$ and $(x,\ell)=\varphi(z)$. 
In rough terms, the homeomorphism $h_t$ on $V$ pushes $t$ along the direction determined by the oriented
submanifolds $\alpha_z$. 
For any $z\in M \setminus \bigcup_{n\in \mathbb Z} f^n(U)$ define $h_t(z)=z$ and, otherwise, define 
$h_t(z)= f^{-n} ( h_{0,t} (f^n(z))  )$ where $n=n(z)\in \mathbb Z$ 
is (unique) so that $f^{n}(z)\in U$. 

By construction $h_0=Id$ and $(h_t)_{t\in [0,1]}$ is a continuous family such that $f\circ h_t = h_t\circ f$ 
(equivalently $h_t\in Z^0(f)$) for every $t\in [0,1]$. In order to complete the proof 
of the theorem we are left to prove that $h_t$ is a homeomorphism for every $t\in [0,1]$.
Notice that $h_t$ is invertible. We proceed to prove the following:

\smallskip
\noindent {\it Claim:} $h_t$ is continuous

\begin{proof}[Proof of the claim:]
 On $M \setminus \overline{\bigcup_{n\in \mathbb Z} f^n(W)}$ we have that $h_t=id$ is clearly continuous. Since 
$h_t$ is also continuous on each open set of the form $f^n(W)$, $n\in \mathbb Z$, it remains to prove the continuity
on points that belong to either the attractor $\Lambda_{i_1}$ or the repeller $\Lambda_{i_s}$. Let $\lambda\in (0,1)$
denote a hyperbolicity constant for $\Omega(f)$.  

Fix $x\in \Lambda_{i_1}$ (if $x\in \Lambda_{i_s}$ the computations are analogous, replacing $f$ by $f^{-1}$)
and $\vep>0$. Pick $0<\delta <\vep/2$ small so that for any point $y\in \bigcup_{n\in \mathbb Z} f^n(W)$ with
$d(x,y)<\delta$, the unique $n_y\in \mathbb Z$ (and necessarily negative)  so that $f^{n_y}(y) \in U$ 
satisfies $\lambda^{-n_y} \zeta < \vep/2$. This holds because the number of iterates necessary for a point 
to enter the set $W$  grows to infinity for points sufficiently close to the attractor $\Lambda_{i_1}$.
With this choice, if $d(x,y)<\delta$ and $y \notin \bigcup_{n\in \mathbb Z} f^n(W)$ then 
$d(h(x),h(y))=d(x,y)<\vep/2.$
Now assume, otherwise, that $d(x,y)<\delta$ and $y \in \bigcup_{n\in \mathbb Z} f^n(W)$.
By construction, $h_{0,t}(f^{n_y}(y)) \in W^s(f^{n_y}(y))$ and $d(h_{0,t}(f^{n_y}(y)), f^{n_y}(y))<\zeta$.
\begin{figure}[htb]
\includegraphics[scale=.50]{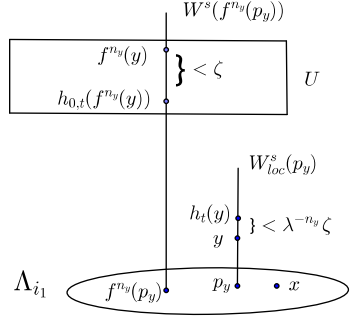}
\caption{Continuity argument}
\end{figure}
Therefore, by triangular inequality together with the uniform contraction along stable leaves,
\begin{align*}
d( h_t(x), h_t(y)) 
	& = d( x, f^{-n_y}\circ h_{0,t} \circ f^{n_y}(y)) \\
	& = d( f^{-n_y} \circ f^{n_y}(x), f^{-n_y} \circ h_{0,t} \circ f^{n_y}(y)) \\
	& \le d(  f^{-n_y} \circ f^{n_y}(x), f^{-n_y} \circ f^{n_y}(y)  ) 
		+ d(   f^{-n_y} \circ f^{n_y}(y), f^{-n_y} \circ h_{0,t} \circ f^{n_y}(y)  ) \\
	& \le d(x, y)
		+ \lambda^{-n_y} d( f^{n_y}(y), h_{0,t} \circ f^{n_y}(y)  ) \\
	& \le \delta
		+ \lambda^{-n_y} \zeta <\vep.
\end{align*}
This proves the continuity of $h_t$ at $x$, and completes the proof of the claim.
\end{proof}

Since the continuity of $h_t^{-1}$ is analogous we conclude that $h_t$ is a homeomorphism $C^0$-close to the identity and that $h_t \in Z^0(f)$ for every $t\in [0,1]$. Therefore, $Z^0(f)$ is not discrete.
This finishes the proof of the theorem.

\section{Some examples}\label{sec:examples}

Our first example is a Morse-Smale diffeomorphism (hence structurally stable), in which case the $C^0$-centralizer contains 
a continuum of homeomorphisms $C^0$-close to identity.

\begin{example}\label{ex:N-S}
Let $f: \mathbb S^1 \to \mathbb S^1$ be a circle diffeomorphism 
so that the non-wandering set is $\Omega(f)=\{N,S\}$, where $N$ denotes the repelling fixed
point and $S$ denotes the contracting fixed point. $f$ is clearly structurally stable.
We claim that $Z^0(f)$ contains a continuum $C^0$-close to the identity. Since $f \mid_{\mathbb S^1 \setminus \{N\}}$
is homeomorphic to a linear contraction on the real line, we may assume without loss of generality that $f(x)=\frac12 x$ and fix
$I= [-1,-\frac12) \cup (\frac12, 1]$ as a fundamental domain. 
For every $x\in \mathbb S^1 \setminus \{N\}$ let  $n=n(x)\in \mathbb Z$ be the unique
integer so that $f^n(x)\in I$. Using that $n(f(x)) = n(x)-1$ for every $x\in \mathbb R$ it is not hard to check that 
every continuous increasing map $h_0: I \to I$ with $h_0 \mid_{\partial I} =id$ determines an element in $Z^0(f)$: the
homeomorphism $h: \mathbb S^1\setminus\{N,S\} \to \mathbb S^1\setminus\{N,S\}$ given by $h(x) = f^{-n(x)} ( h_0 (f^{n(x)}(x))  )$ 
satisfies
\begin{align*}
h ( f (x ) )
	& =  f^{-n(f(x))} ( h_0 (f^{n(f(x))}(f(x)))  ) 
	 = f^{-n(x)+1} ( h_0 (f^{n(x)-1}(f(x)))  ) \\
	& = f( f^{-n(x)} ( h_0 (f^{n(x)}(x)  ))) 
	= f (h(x)) 
\end{align*}
for every $\mathbb S^1\setminus\{N,S\}$, and extends continuously to $\mathbb S^1$ by $h(N)=N$ and $h(S)=S$.
As a consequence of Theorem~\ref{prop:bijection}, for every $g\in \cC_f$ there exists a continuum of homeomorphisms 
conjugating $f$ and $g$ and  that are $C^0$-close to the identity. Moreover, for any $\vep>0$ there exists a ball $B$ 
of radius $\vep$ so that $\diam f^n(B)<\vep$ for all $n\in \mathbb Z$. This implies that the set of points of $\vep$-expansiveness are not dense in $\mathbb S^1$.
\end{example}

In the next example we exhibit an open set of densely expansive structurally stable diffeomorphisms with non-discrete $C^0$-centralizer.

\begin{example}\label{ex:2}
Let $f : \mathbb S^1 \to \mathbb S^1$ be the Morse-Smale diffeomorphism from Example~\ref{ex:N-S}
and let $g : \mathbb T^2 \to \mathbb T^2$ be an Anosov diffeomorphism. 
It is clear that the diffeomorphism $f\times g : \mathbb S^1 \times \mathbb T^2 \to \mathbb S^1 \times \mathbb T^2$ is an Axiom A $C^1$-diffeomorphism, the non-wandering set consists of two non-trivial basic pieces and that it 
satisfies the strong transversality condition. Moreover, $Z^0(f\times g) \supset Z^0(f) \times \{id\}$, thus it
is not discrete. Then, Theorem~\ref{prop:bijection} guarantees there exists a continuum of conjugacies 
$C^0$-close to the identity between $f\times g$ and any $C^1$-diffeomorphism $F\in \cC_{f\times g}$. 
By structural stability, Theorem~\ref{prop:bijection} implies that the later holds for a $C^1$-open neighborhood
of the diffeomorphism $f\times g$. 
In contrast to Example~\ref{ex:N-S}, the non-wandering set is formed by one attractor and one repeller, both non-trivial. 
By Theorem~\ref{prop:expansive.points}, every $C^1$-small perturbation of $f\times g$ has sensitivity
to initial conditions.
Finally, observe that  the set of points with expansiveness are dense in $\mathbb S^1 \times \mathbb T^2$
while the centralizer is not discrete. This shows that dynamics with a dense subset of expansive points may have non-discrete
$C^0$-centralizer and proves that Walters' lemma is optimal (cf. Lemma~\ref{thm:C0-homoclinic}).
Moreover, since this is a partially hyperbolic diffeomorphism with a one-dimensional central bundle 
then it is  also entropy expansive (see \cite{DFPV} for definition and proof).
\end{example}

The following example shows that $C^0$ and $C^1$-centralizers can be both discrete but distinct for a locally $C^1$-Baire generic set of $C^1$-diffeomorphisms.

\begin{example}
Let $f$ be a $C^1$-Anosov diffeomorphism on the two torus $\mathbb T^2$ with a unique fixed point. Then there 
exists a $C^1$-open neighborhood $\mathcal U$ of $f\in \text{Diff}^{\,1}(\mathbb T^2)$ formed subset of Anosov diffeomorphisms, 
topologically conjugate to $f$. The $C^0$-centralizer of $f$ is discrete but non-trivial (it has two generators) 
cf.~\cite{Ro05}.
Since every $C^1$-diffeomorphism $g\in \mathcal U$ is topologically conjugate to $f$ then it follows from
Theorem~\ref{prop:bijection} that the $C^0$-centralizer of $g$ is discrete but non-trivial.
Now, we recall that there exists a $C^1$-Baire residual subset $\mathcal R \subset \cU$ 
diffeomorphisms so that $Z^1(g)$ is trivial for every $g\in \mathcal R$ (cf.~\cite{BCW}).
Thus there exists a $C^1$-Baire residual subset $\mathcal R_1 \subset \cU$ so that 
$
Z^0(g) \supsetneq Z^1(g) =\{g^n \colon n\in \mathbb Z\}
	\quad \text{for every} \; g\in \mathcal R_1.
$ 
In particular, the analogous affirmative statement to Smale's question is no longer true for $C^0$-centralizers.
\end{example}

We also derive a consequence for reversible dynamics. 

\begin{example}
Given an involution  $R: M \to M$  (ie. $R^2=id$) let 
$\text{Homeo}_R(M)$ denote the set of homeomorphisms $f$ so that $R\circ f = f^{-1} \circ R$. 
These are called \emph{$R$-reversible homeomorphisms}. Notice that $f\in \text{Homeo}_R(M)$ if and only if
$f^{-1}$ belongs to the conjugacy class of $f\in \Homeo$ and that $R \in \cH_{f,f^{-1}}$. 
Consider the (non-empty) set
$
\{ h\in \Homeo \colon h\circ f = f^{-1}\circ h \}
$
of conjugacies between $f$ and $f^{-1}$. 
By Theorem~\ref{prop:bijection}, the map
$F_R \colon  Z^0(f)  \to  \mathcal H_{f,f^{-1}}$ given by 
$F_R( \tilde f) =R \circ \tilde f$ is a homeomorphism. Thus, if  $f\in \text{Homeo}_R(M)$ then 
$Z^0(f)$ is discrete if and only if  there exists $\vep>0$ so that there is no conjugacy between 
$f$ and $f^{-1}$, distinct from $R$ that is $\vep$-$C^0$-close to $R$.
Clearly, both properties hold for reversible Anosov diffeomorphisms.
Moreover, we deduce that if $Z^0(f)$ is trivial then every conjugation 
between $f$ and $f^{-1}$ is of the form $R\circ f^{n}$ for some $n\in \mathbb Z$.
\end{example}

\medskip
\subsection*{Acknowledgements} 
The authors are grateful to V. Ara\'ujo for useful comments and to the anonymous referee for valuable comments that really helped to improve the manuscript.
Most of this work was developed during a visit of the second author to University of Porto, whose excellent research
conditions are greatly acknowledged. J.R. and P.V. were partially supported by
 by CMUP \-(UID/\-MAT/\-00144/2013), which is funded by FCT (Portugal) with national (MEC) and European 
 structural funds through the programs FEDER, under the partnership agreement PT2020.
P.V. was partially supported by CNPq-Brazil.


\end{document}